%% file: main_final.tex
\documentclass[journal,twoside,web]{ieeecolor}
\usepackage{lcsys}
\usepackage{cite}
\usepackage{amsmath,amssymb,amsfonts}
\usepackage{graphicx}
\usepackage{textcomp}

\usepackage{balance}

\usepackage[english]{babel}
\usepackage{mathtools}
\usepackage[latin1,utf8]{inputenc}
\usepackage{epstopdf}
\usepackage{subcaption}
\usepackage{algorithm}
\usepackage{algpseudocode}
\usepackage{bm}
\usepackage[subpreambles]{standalone}
\usepackage{comment}
\usepackage{stmaryrd}
\usepackage{accents}
\usepackage{float}

\usepackage[normalem]{ulem}

\usepackage{pgfplots}
\pgfplotsset{compat=1.18}
\pgfplotsset{every axis/.append style = {
                    label style = {font=\scriptsize},
                    tick label style = {font=\scriptsize},
                    legend style = {font=\scriptsize}
                }
            }
\usepgfplotslibrary{groupplots}
\usepgfplotslibrary{fillbetween}

\makeatletter
\renewcommand{\ALG@beginalgorithmic}{\small}
\makeatother

\def\minwrt[#1]{\underset{#1}{\text{minimize }}}
\def\argminwrt[#1]{\underset{#1}{\text{arg min }}}
\def\maxwrt[#1]{\underset{#1}{\text{maximize }}}
\def\argmaxwrt[#1]{\underset{#1}{\text{arg max }}}
\def\maxemphwrt[#1]{\underset{#1}{\text{\emph{maximize} }}}

\DeclareMathOperator*{\minimize}{minimize}
\DeclareMathOperator*{\maximize}{maximize}
\DeclareMathOperator*{\subjectto}{subject\:to}

\newcommand{\rarrow}[1]{\accentset{\rightharpoonup}{#1}}
\newcommand{\larrow}[1]{\accentset{\leftharpoonup}{#1}}

\newcommand{\R}{\mathbb{R}}
\newcommand{\N}{\mathbb{N}}

\graphicspath{{./figures/}}

\newtheorem{theorem}{Theorem}
\newtheorem{remark}{Remark}
\newtheorem{proposition}{Proposition}

\definecolor{tol_vib_orange}{HTML}{EE7733}
\definecolor{tol_vib_blue}{HTML}{0077BB}
\definecolor{tol_vib_cyan}{HTML}{33BBEE}
\definecolor{tol_vib_magenta}{HTML}{EE3377}
\definecolor{tol_vib_red}{HTML}{CC3311}
\definecolor{tol_vib_teal}{HTML}{009988}
\definecolor{tol_vib_gray}{HTML}{BBBBBB}

\title{\huge
    Fast computation of the TGOSPA metric for multiple target tracking via unbalanced optimal transport
}

\author{%
    Viktor Nevelius Wernholm$^*$,
    Alfred Wärnsäter$^*$, 
    and Axel Ringh
    \thanks{
        This work was partially supported by the Wallenberg AI, Autonomous Systems and Software Program (WASP) funded by the Knut and Alice Wallenberg Foundation, Sweden, by 
        the Swedish Research Council (VR) under grant 2020-03454, and by KTH Digital Futures.
    }%
    \thanks{
            V.~Nevelius Wernholm is with Saab Surveillance, Saab AB, SE–412 89 Gothenburg, Sweden
        {\tt\scriptsize viktor.neveliuswernholm@saabgroup.com}
    }%
    \thanks{
        A.~Wärnsäter is with the Department of Mathematics, KTH Royal Institute of Technology, SE-100 44 Stockholm, Sweden
        {\tt\scriptsize alfwar@kth.se}
    }%
    \thanks{
        A.~Ringh is with the Department of Mathematical Sciences, Chalmers University of Technology and University of Gothenburg, SE-412 96 Gothenburg, Sweden 
        {\tt\scriptsize axelri@chalmers.se}
    }%
}

\begin{document}

% To see on which pages there are extra space
%\raggedbottom

% To allow equations to break over multiple pages
%\allowdisplaybreaks

\maketitle
\thispagestyle{empty}

%%%%%%%%%%%%%%%%%%%%%%%%%%%%%%%%%%%%%%%%%%%%%%%%%%%%%%%%%%%%%%%%%%%%%%%%%%%%%%%%%%%%%%%%%%%%%%%%%%%%%%%%%%%%%%%%%%%%%%%%

\begin{abstract}
In multiple target tracking, it is important to be able to evaluate the performance of different tracking algorithms. The trajectory generalized optimal sub-pattern assignment metric (TGOSPA) is a recently proposed metric for such evaluations. The TGOSPA metric is computed as the solution to an optimization problem, but for large tracking scenarios, solving this problem becomes computationally demanding. In this paper, we present an approximation algorithm for evaluating the TGOSPA metric, based on casting the TGOSPA problem as an unbalanced multimarginal optimal transport problem. Following recent advances in computational optimal transport, we introduce an entropy regularization and derive an iterative scheme for solving the Lagrangian dual of the regularized problem. Numerical results suggest that our proposed algorithm is more computationally efficient than the alternative of computing the exact metric using a linear programming solver, while still providing an adequate 
approximation of the metric.
\end{abstract}

\begin{IEEEkeywords}
Estimation, Numerical algorithms, Optimization, Optimization algorithms 
\end{IEEEkeywords}

%%%%%%%%%%%%%%%%%%%%%%%%%%%%%%%%%%%%%%%%%%%%%%%%%%%%%%%%%%%%%%%%%%%%%%%%%%%%%%%%%%%%%%%%%%%%%%%%%%%%%%%%%%%%%%%%%%%%%%%%

\section{Introduction} \label{sec:introduction}

\renewcommand{\thefootnote}{\fnsymbol{footnote}}
\setcounter{footnote}{1}
\makeatletter\def\Hy@Warning#1{}\makeatother
\footnotetext{Equal contribution.}
\renewcommand{\thefootnote}{\arabic{footnote}}
\setcounter{footnote}{0}

\IEEEPARstart{M}{ultiple}
%Multiple
target tracking (MTT) deals with the task of estimating targets that appear, disappear, and move within a scene, given data from noisy measurements.
A wide range of algorithms that solves this task has been developed, see, e.g.,
\cite{fortmann_sonar_1983, blackman_design_1999, blackman_multiple_2004}.
To objectively evaluate the performance of different MTT algorithms in test scenarios, where the ground truth trajectories of the objects are known, one needs a distance function that quantifies the error between the estimated target trajectories and the ground truth trajectories.
A recently proposed metric
for evaluating MTT algorithms is the trajectory generalized optimal sub-pattern assignment (TGOSPA) metric \cite{garcia-fernandez_metric_2020, garcia-fernandez_time-weighted_2021, krejčí2024tgospametricparametersselection}. This is an extension of the GOSPA metric \cite{rahmathullah_generalized_2017}. The latter is a metric between snapshots of ground truth and identified targets in a single time frame, and it penalizes localization errors for properly detected targets, missed targets in the ground truth, and falsely detected targets that do not exist in the ground truth. The TGOSPA metric generalizes this to a metric for tracks over multiple time frames by also including a penalty for so-called track switching, where the identities of ground truth targets erroneously get swapped.

Both the GOSPA metric and the TGOSPA metric are formulated as optimization problems. While the GOSPA metric involves solving an assignment problem, which can be done efficiently using, e.g.,
the Hungarian algorithm (see \cite[Chp.~11]{papadimitriou1982combinatorial}),
computation of the TGOSPA metric requires solving a set of coupled, consecutive assignment problems. 
Computing the TGOSPA metric is therefore in general only tractable when the tracking scenario contains a small number of targets \cite[Sec.~III.C]{garcia-fernandez_metric_2020}.
Therefore, \cite{garcia-fernandez_metric_2020, garcia-fernandez_time-weighted_2021} also suggests a linear programming (LP) relaxation of the TGOSPA metric, and this relaxation is in fact also a metric. Nevertheless, for large tracking scenarios, solving the corresponding LP still requires significant computational resources.

In this paper, which is based on the master's thesis \cite{nevelius2024efficient}, we derive an efficient method for approximately solving the LP-relaxed TGOSPA problem.
This is done by first casting the problem as an unbalanced multimarginal optimal transport problem
and deriving the Sinkhorn iterations
\cite{cuturi_sinkhorn_2013, benamou_iterative_2014, peyre2019computational, haasler_multimarginal_2021, ringh_graph-structured_2022, haasler_scalable_2023}.
However, in the multimarginal setting, the Sinkhorn iterations only partially alleviate the computational cost, as computing the projections needed in the iterations in general scales exponentially with the number of marginals. This can be overcome if, e.g., the cost function has a graph structure \cite{benamou_iterative_2014, haasler_multimarginal_2021, ringh_graph-structured_2022, haasler_scalable_2023}, but that is not the case for the problem considered here.
Nevertheless, the cost function in the problem is a decomposable structured tensor, and using the structure we develop a computationally efficient way to evaluate the projections needed. Finally,
we prove that the algorithm converges linearly, and demonstrate its performance on a number of examples.

%%%%%%%%%%%%%%%%%%%%%%%%%%%%%%%%%%%%%%%%%%%%%%%%%%%%%%%%%%%%%%%%%%%%%%%%%%%%%%%%%%%%%%%%%%%%%%%%%%%%%%%%%%%%%%%%%%%%%%%%

\section{Background} \label{sec:background}

In this section, we introduce the TGOSPA metric and the basics of entropy regularized multimarginal optimal transport. The section is also used to set up notation.
To this end, we use $^\top$ to denote the transpose of a vector or a matrix, $\R_+$ to denote non-negative real numbers, and the operation $\exp$ on a vector, matrix, or tensor means elementwise exponential.

\subsection{TGOSPA} \label{sub_sec:t_gospa}

A trajectory $X$ on time steps $\{1, \dots, T\} \subset \N$ is defined
as the sequence $(\bm{x}^1, \dots, \bm{x}^T)$ of sets $\bm{x}^t$.
If a trajectory is alive and in a state $x^t \in \R^N$ at time step $t$, then $\bm{x}^t = \{ x^t \}$. Otherwise, $\bm{x}^t = \varnothing$.
Let $\mathcal{T}$ denote the set of all possible such trajectories. The TGOSPA metric is a metric between sets of trajectories, i.e., a function that maps
$\mathcal{T} \times \mathcal{T} \to \R_+$, and it can be formulated as an
integer linear program. To this end,
let $\mathbf{X}, \mathbf{Y} \in \mathcal{T}$, and denote with $\bm{x}_i^t$ and $\bm{y}_j^t$ the (possibly empty) state of ground truth $i$ and estimate $j$ at time step $t$. For a set of ground truths $\mathbf{X}$ consisting of $m$ trajectories, and a set of estimates $\mathbf{Y}$ consisting of $n$ trajectories, we can represent their associations in time step $t$ as an $(m + 1) \times (n + 1)$ binary matrix $W^t$. For such a matrix, $W_{i,j}^t = 1$ means that ground truth $i$ and estimate $j$ are assigned to each other at time step $t$, and $W_{i,j}^t = 0$ means that they are not. Every such matrix thus satisfies
\begin{subequations} \label{eq:assignment_matrices}
\begin{alignat}{3}
    & W^t_{(i,j)} \in \{0, 1\},          & \quad i = 1, \dots, m, & \quad j = 1, \dots, n,
        \label{eq:assignment_matrices_binary} \\
    & \sum_{i=1}^{m+1} W^t_{(i,j)} = 1,  & \quad j = 1, \dots, n, &
        \label{eq:assignment_matrices_targets} \\
    & \sum_{j=1}^{n+1} W^t_{(i,j)} = 1,  & \quad i = 1, \dots, m. &
        \label{eq:assignment_matrices_estimates}
\end{alignat}
\end{subequations}
Here, \eqref{eq:assignment_matrices_targets} implies that each ground truth trajectory is either assigned to exactly one estimated trajectory or unassigned (in the latter case, $W^t_{(m+1,j)} = 1$), and \eqref{eq:assignment_matrices_estimates} implies the corresponding property among the estimated trajectories. 

Let $\mathcal{W}_{(m, n)}$ be the set of assignment matrices described by \eqref{eq:assignment_matrices}. The TGOSPA metric
is the power $1/p$ of
\begin{equation} \label{eq:t-gospa}
\begin{multlined}[0.9\columnwidth]
    \minimize_{\substack{W^t \in \mathcal{W}_{(m, n)}, \\ t = 1, \dots, T}}
        \sum_{t=1}^T \sum_{i=1}^{m+1} \sum_{j=1}^{n+1} D^t_{(i,j)} W^t_{(i,j)} \\
        + \frac{\gamma^p}{2} \sum_{t=1}^{T-1} \sum_{i=1}^m \sum_{j=1}^n \left| W^{t+1}_{(i,j)} - W^t_{(i,j)} \right|,
\end{multlined}
\end{equation}
see \cite[Lem.~1]{garcia-fernandez_metric_2020}.
Here, $1 \leq p < \infty$
determines to what extent outliers are penalized, $\gamma > 0$
determines how much we penalize track-switches, and
\begin{equation*}
    (D^t_{(i,j)})^{1/p} =
    \begin{cases}
        \min\left( || x^t_i - y^t_j ||, c \right)\!, \! & \text{if} \: \bm{x}_i = \{ x_i \}, \bm{y}_j = \{ y_j \}, \\
        0,                                            & \text{if} \: \bm{x} = \varnothing, \bm{y} = \varnothing, \! \\
        c / 2^{1 / p}    \!                             & \text{else},
    \end{cases}
\end{equation*}
where $c > 0$ is a cut-off parameter, and where $\bm{x}^t_{m+1} = \varnothing$ and $\bm{y}^t_{n+1} = \varnothing$ for all $t$.

\begin{remark}
For $t=1,\dots,T$, the problem
\[
\minimize_{W^t \in \mathcal{W}_{(m, n)}} \sum_{i=1}^{m+1} \sum_{j=1}^{n+1} D^t_{(i,j)} W^t_{(i,j)}
\]
in \eqref{eq:t-gospa}
is an unbalanced optimal transport problem \cite{georgiou2008metrics, beier2023unbalanced}.  
\end{remark}

Problem \eqref{eq:t-gospa} is an integer linear program, and to reduce the computational cost, \cite{garcia-fernandez_metric_2020} also proposes the relaxed TGOSPA problem obtained by relaxing
the binary constraints \eqref{eq:assignment_matrices_binary} to $W^t_{(i,j)} \geq 0$.%
\footnote{Note that $W^t_{(i,j)} \leq 1$ is implicitly enforced by \eqref{eq:assignment_matrices_targets} and \eqref{eq:assignment_matrices_estimates}. Also note that the LP relaxation is known to not be tight, see \cite[Prop.~4]{nevelius2024efficient}.}
It can be shown that this relaxed TGOSPA problem also defines 
a metric on $\mathcal{T}$, see \cite[Sec.~IV]{garcia-fernandez_metric_2020}.

\subsection{Multimarginal Optimal Transport}\label{sec:MMOT}

Optimal transport deals with problems of how mass can be moved between an initial distribution and a target distribution as efficiently as possible. Such problems can be extended to multimarginal optimal transport problems, where an optimal transport plan between several distributions is sought.

Let $\hat{M} \in \R_+^{N_1 \times \ldots \times N_T}$ denote the transport tensor, $\hat{C} \in \R_+^{N_1 \times \ldots \times N_T}$ the cost tensor, and $\hat{\mu}_t \in \R_+^{N_t}$, for $t = 1, \ldots, T$, the marginals
of a multimarginal optimal transport problem. Additionally, for $t = 1, \dots, T$,
define the projections
$\hat{P}_{t}(\hat{M})_{i_t} = \sum_{i_1,\dots,i_{t-1},i_{t+1},\dots,i_T} \hat{M}_{i_1,\dots,i_T}$.
The most common type of multimarginal optimal transport problems is then
\begin{equation} \label{eq:mot}
\begin{aligned} 
    \minimize_{\hat{M} \in \R_+^{N_1 \times \ldots \times N_T}} \quad & \langle \hat{C}, \hat{M} \rangle \\
    \subjectto\hspace{3mm}                 \quad & \hat{P}_t(\hat{M}) = \hat{\mu}_t, \quad t = 1, \dots, T,
\end{aligned}    
\end{equation}
where
    $
    \langle A, B \rangle \coloneqq \sum_{i_1, \dots, i_L} A_{i_1, \dots, i_L} B_{i_1, \dots, i_L},
    $
for tensors $A$ and $B$ with $L$ indices.

Even though \eqref{eq:mot} is an LP, it is, in general, difficult to solve directly for larger instances since the number of variables 
increases exponentially with $T$.
This makes manipulating and storing the tensors $\hat{C}$ and $\hat{M}$ computationally infeasible.

A first step towards addressing this is to compute approximate solutions using entropy regularization
\cite{cuturi_sinkhorn_2013, benamou_iterative_2014, peyre2019computational, haasler_multimarginal_2021, ringh_graph-structured_2022, haasler_scalable_2023}.
To this end, for a tensor $A$ with $L$ indices, let the entropy of $A$ be defined by
\begin{equation*}
    E(A) = \sum_{i_1,\dots,i_L}
        (A_{i_1,\dots,i_L} \log(A_{i_1,\dots,i_L}) - A_{i_1,\dots,i_L} + 1).
\end{equation*}
For some regularization parameter $\varepsilon > 0$,
adding the term $\varepsilon E(\hat{M})$ to the objective function of \eqref{eq:mot} yields the regularized problem
\begin{equation} \label{eq:mot_reg}
\begin{aligned} 
    \minimize_{\hat{M} \in \R^{N^T}} \quad & \langle \hat{C}, \hat{M} \rangle + \varepsilon E(\hat{M}) \\
    \subjectto                 \quad & \hat{P}_t(\hat{M}) = \hat{\mu}_t, \quad t = 1, \dots, T.
\end{aligned}    
\end{equation}
A Sinkhorn algorithm for solving this problem is obtained by doing block coordinate ascent in the Lagrangian dual problem.
However, for such an algorithm to be efficient, one must also find an efficient way to compute the projections needed in the algorithm; cf.~\cite{benamou_iterative_2014, haasler_multimarginal_2021, ringh_graph-structured_2022, haasler_scalable_2023}.

%%%%%%%%%%%%%%%%%%%%%%%%%%%%%%%%%%%%%%%%%%%%%%%%%%%%%%%%%%%%%%%%%%%%%%%%%%%%%%%%%%%%%%%%%%%%%%%%%%%%%%%%%%%%%%%%%%%%%%%%

\section{A Novel Algorithm for Approximation of the Relaxed TGOSPA Metric} \label{sec:novel_algorithm}

In this section,
we derive a Sinkhorn algorithm for approximating the optimal value of the LP-relaxed version of the TGOSPA metric. This is done in three main steps.
First, we reformulate the LP-relaxed version of \eqref{eq:t-gospa} as a multimarginal optimal transport problem over a high-order tensor.
Second, following the steps in the literature, we derive a block coordinate ascent algorithm in the dual to the entropy regularized version of this problem.
Third, we show that the structure of the cost function can be used to efficiently compute the projections needed in the algorithm.
Using this, we then formulate the full algorithm, and also show that it converges linearly.

\subsection{TGOSPA as Multimarginal Optimal Transport} \label{sub_sec:t_gospa_as_mot}

Let $\mathcal{M}_+$ denote the set of all non-negative tensors with $2T$ indices such that odd indices have
dimension $m + 1$, and even indices have dimension $n + 1$. We index such tensors using the notation
$M_{(i_1, j_1), \dots, (i_T, j_T)}$ for $i_1, \dots, i_T \in \{1, \dots, m+1\}$ and
$j_1, \dots, j_T \in \{1, \dots, n+1\}$,
and the element
$M_{(i_1, j_1), \dots, (i_T, j_T)}$ should be interpreted as the amount of mass transported along the trajectory
$(i_1, j_1), \dots, (i_T, j_T)$.
Analogously to the projections defined in Section~\ref{sec:MMOT}, let
\begin{equation*}
    P_t(M)_{(i_t, j_t)} =
        \sum_{(i_1, j_1), \dots, (i_T, j_T) \setminus (i_t, j_t)}
        M_{(i_1, j_1), \dots, (i_T, j_T)},
\end{equation*}
for $t = 1, \dots, T$, and let
\begin{equation*}
\begin{multlined}[\columnwidth]
    P_{t, t+1}(M)_{(i_t,j_t), (i_{t+1}, j_{t+1})} = \!\!\!
        \sum_{\substack{(i_1, j_1), \dots, (i_T, j_T) \setminus \\  (i_t, j_t), (i_{t+1}, j_{t+1})}} \!\!
        M_{(i_1, j_1), \dots, (i_T, j_T)},
\end{multlined}
\end{equation*}
for $t = 1, \dots, T - 1$. Here, $P_t(M)_{(i_t, j_t)}$ should be interpreted as the amount of mass at position
$(i_t, j_t)$ at time step $t$, and $P_{t, t+1}(M)_{(i_t, j_t), (i_{t+1}, j_{t+1})}$ should be interpreted as the mass
transported between position $(i_t, j_t)$ and position $(i_{t+1}, j_{t+1})$ from time step $t$ to $t+1$. 

Next, note that the assignment matrices $W^t$ can be naturally identified with the projections $P_t(M)$ for $t = 1, \dots, T - 1$. Using this idea,
we prove the following.
\begin{theorem}\label{thm:tgospa_as_mot}
    Let $\bar{\mu} \in \R^{m + 1}$ and $\Tilde{\mu} \in \R^{n + 1}$ be
    defined by $\bar{\mu} = \left( \bm{1}_m^\top, n \right)^\top$ and $\tilde{\mu} = \left( \bm{1}_n^\top, m \right)^\top$, where $\bm{1}_\ell \in \R^\ell$
    is a vector with all elements equal to one. Further, let
    \begin{equation*}
    F_{(i,j), (k,l)} =
    \begin{dcases}
        \frac{\gamma^p}{2} (1-\delta_{ik} \delta_{jl}) \big((1 - \delta_{k, m+1})(1 - \delta_{l, n+1}) \\
        \quad + (1 - \delta_{i, m+1})(1 - \delta_{j, n+1})\big), \: \text{if $i = k$}, \\
        \infty, \quad \text{otherwise},
    \end{dcases}
    \end{equation*}
    where $\delta$ denotes the Kronecker delta, i.e. $\delta_{i,j} = 1$ if $i = j$ and $0$ otherwise.
    Then the multimarginal optimal transport problem
    \begin{subequations}\label{eq:t-gospa_mot}
    \begin{equation}\label{eq:t-gospa_mot-problem}
    \begin{aligned}
        \minimize_{M \in \mathcal{M}_+} \quad & \langle C, M \rangle, \\
        \subjectto                            \quad & P_t(M) \mathbf{1}_{n+1} = \bar{\mu}, \: t = 1, \dots, T, \\
                                                    & P_t(M)^\top \mathbf{1}_{m+1} = \tilde{\mu}, \: t = 1, \dots, T,
    \end{aligned}
    \end{equation}
    where
    \begin{equation}\label{eq:t-gospa_mot-cost}
        \!\! C_{(i_1, j_1), \dots, (i_T, j_T)} =
        \sum_{t=1}^T D^t_{(i_t,j_t)} \! + \! \sum_{t=1}^{T-1} F_{(i_t,j_t), (i_{t+1},j_{t+1})},
    \end{equation}
    \end{subequations}
    and the LP-relaxed version of \eqref{eq:t-gospa} has the same optimal value.
\end{theorem}

\begin{proof}
    To show that the optimal values coincide, we first claim that for any feasible solution $M$ of
    \eqref{eq:t-gospa_mot}, there exists a feasible solution to the LP-relaxed version \eqref{eq:t-gospa} with a lower or
    equal objective value. To see this, let
    $W^t = P_t(M)$ for $t = 1, \dots, T$.
    From the constraints in \eqref{eq:t-gospa_mot}, it follows that $\{W^t\}_t$ is a feasible solution
    to the LP relaxation of \eqref{eq:t-gospa}.
    Next, a direct calculation gives that 
    \begin{align}
        \!\!\!\!\!\!\!\! \langle F, P_{t,t+1}(M) \rangle \!
        &\ge \! \frac{\gamma^p}{2} \! \sum_{i,j=1}^{m, n} \! \left[ W^t_{(i,j)} \! + \! W^{t+1}_{(i,j)} \! - \! 2 \min\left(\!W^t_{(i,j)}, W^{t+1} _{(i,j)} \right) \! \right] \nonumber \\
        &= \frac{\gamma^p}{2} \sum_{i=1}^m \sum_{j=1}^n \left| W^{t+1}_{(i,j)} - W^t_{(i,j)} \right|, \label{eq:tgospa_ineq}
    \end{align}
    where the inequality is due to $P_{t,t+1}(M)_{(i,j), (i,j)} \le W^s_{ij}$ for $s = t, t+1$, and where the last equality follows from the identity
    $a + b - 2 \min(a, b) = |a - b|$ for $a, b \in \R$. Using \eqref{eq:tgospa_ineq}, we get
    $\langle C, M \rangle
        = \sum_{t=1}^T \langle D^t, P_t(M) \rangle + \sum_{t=1}^{T-1} \langle F, P_{t,t+1}(M) \rangle \ge \sum_{t=1}^T \sum_{i=1}^{m+1} \sum_{j=1}^{n+1} D^t_{(i,j)} W^t_{(i,j)}
        + \frac{\gamma^p}{2} \sum_{t=1}^{T-1} \sum_{i=1}^m \sum_{j=1}^n \left| W^{t+1}_{(i,j)} - W^t_{(i,j)} \right|$,
    which proves the claim. Conversely, let $\{W^t\}_t$ be a feasible solution to the LP-relaxed version of \eqref{eq:t-gospa}. From \eqref{eq:tgospa_ineq}, we see that a sufficient condition for equality 
    is that $M$ is feasible to \eqref{eq:t-gospa_mot}, $P_t(M) = W^t$ for $t = 1, \dots, T$, and that $M$ is such that $P_{t,t+1}(M)_{(i,j), (i,j)} = \min\left(W^t_{(i,j)}, W^{t+1}_{(i,j)}\right)$ for $t = 1, \dots, T-1$,
    $i = 1, \dots, m$ and $j = 1, \dots, n$. Such $M$ exists as it simply corresponds to a transport plan where as much mass as possible is unmoved.
\end{proof}

\begin{figure}
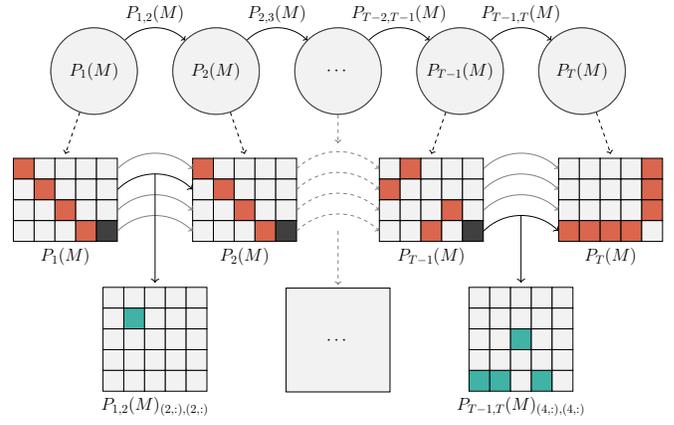

    \centering
    \resizebox{\columnwidth}{!}{
        \includestandalone{figures/structure_figure}
    }
    \caption{Illustration of the problem structure. The top row shows the high level structure, the middle row the assignment
    matrices, and the bottom row the flows between rows in the assignment matrices.}
    \label{fig:structure}
\end{figure}

In Figure \ref{fig:structure}, the structure of problem \eqref{eq:t-gospa_mot} is shown. The first part of the cost tensor \eqref{eq:t-gospa_mot-cost} corresponds
to costs on the assignment matrices directly (middle row), while the second part corresponds to transports between
them (bottom row).

\begin{remark}
The marginals
$\bar{\mu}$ and $\tilde{\mu}$
have equal mass, which
is a prerequisite for utilizing the optimal transport
framework. Introducing the last elements in $\bar{\mu}$ and $\tilde{\mu}$ is one way to handle unbalanced
transport problems \cite{georgiou2008metrics, beier2023unbalanced}.
\end{remark}

\begin{remark}
    There are multiple ways in which the LP-relaxed version of \eqref{eq:t-gospa} can be formulated as an optimal transport problem. For another formulation, see \cite[Sec.~4.1]{nevelius2024efficient}.
\end{remark}

\subsection{Block Coordinate Ascent in the Lagrangian Dual}
\label{sub_sec:block_coord_asc}

To derive a Sinkhorn
algorithm for \eqref{eq:t-gospa_mot}, we first add an entropy regularization to the objective function, which yields
\begin{subequations}\label{eq:t-gospa_mot_reg}
\begin{align} 
    \minimize_{M \in \mathcal{M}_+} \quad & \langle C, M \rangle + \varepsilon E(M), \label{eq:t-gospa_mot_reg_obj} \\
    \subjectto                            \quad & P_t(M) \mathbf{1}_{n+1} = \bar{\mu}, \: t = 1, \dots, T,
                                                      \label{eq:t-gospa_mot_reg_row_con} \\
                                                & P_t(M)^\top \mathbf{1}_{m+1} = \tilde{\mu}, \: t = 1, \dots, T.
                                                      \label{eq:t-gospa_mot_reg_col_con}
\end{align}
\end{subequations}
For small values of $\varepsilon$, the optimal value of \eqref{eq:t-gospa_mot_reg} is close to the optimal value of \eqref{eq:t-gospa_mot} (cf.~\cite[Sec.~4.5]{peyre2019computational}).
Next, following similar steps as in the literature (see, e.g., \cite{cuturi_sinkhorn_2013, benamou_iterative_2014, peyre2019computational, haasler_multimarginal_2021, ringh_graph-structured_2022, haasler_scalable_2023}), we Lagrangian relax the constraints \eqref{eq:t-gospa_mot_reg_row_con} and \eqref{eq:t-gospa_mot_reg_col_con}
with the dual variables $\bar{\lambda}^t \in \R^{m+1}$ and $\tilde{\lambda}^t \in \R^{n+1}$, respectively. Taking the derivative of the Lagrangian with respect to $M$ and setting it equal to zero gives
\begin{equation} \label{eq:primal_sol}
\begin{aligned}
    M^{(\bar{\lambda},\tilde{\lambda})}_{(i_1, j_1), \dots, (i_T, j_T)}
        &= \exp \left( -\frac{C_{(i_1, j_1), \dots, (i_T, j_T)}}{\varepsilon} \right) \\
        & \quad {} \cdot \prod_{t = 1}^T \exp \left( \frac{\bar{\lambda}^t_{i_t}}{\varepsilon} \right)
            \prod_{t = 1}^T \exp \left( \frac{\tilde{\lambda}^t_{j_t}}{\varepsilon} \right),
\end{aligned}    
\end{equation}
which, in turn, gives the dual problem
\begin{align} \label{eq:t_gospa_dual_reg}
    \maximize_{\substack{\bar{\lambda}^t \in \R^{m+1}, \, \tilde{\lambda}^t \in \R^{n+1} \\ t = 1, \dots, T}}
        \phi(\bar{\lambda}, \tilde{\lambda}),
\end{align}
where $\phi(\bar{\lambda}, \tilde{\lambda}) = \tilde{\phi}(\bar{\lambda}, \tilde{\lambda})+ \varepsilon (m+1)^T (n+1)^T$, and where
\begin{align*}
    & \tilde{\phi}(\bar{\lambda}, \tilde{\lambda})
        = -\varepsilon \sum_{(i_1,j_1), \dots, (i_T,j_T)}
            \exp \left( -\frac{C_{(i_1, j_1), \dots, (i_T, j_T)}}{\varepsilon} \right) \\
            & \cdot 
                \prod_{t=1}^T \! \exp \! \left( \! \frac{\bar{\lambda}^t_{i_t}}{\varepsilon} \! \right) \!
                \prod_{t=1}^T \! \exp \! \left( \! \frac{\tilde{\lambda}^t_{j_t}}{\varepsilon} \! \right)
                \! + \! \sum_{t=1}^T \sum_{i=1}^{m+1} \bar{\mu}_i \bar{\lambda}^t_i
                + \! \sum_{t=1}^T \sum_{j=1}^{n+1} \tilde{\mu}_j \tilde{\lambda}^t_j.
\end{align*}
We have the following result regarding the primal-dual problems \eqref{eq:t-gospa_mot_reg} and \eqref{eq:t_gospa_dual_reg}.
\begin{proposition}\label{thm:strong_duality}
    \textcolor{black}{The optimization problems \eqref{eq:t-gospa_mot_reg} and \eqref{eq:t_gospa_dual_reg} are both convex, both have optimal solutions, and they attain the same optimal value. Moreover, the optimal solutions to the two problems are related via \eqref{eq:primal_sol}.}
\end{proposition}
\begin{proof}
    Due to space limitations, only a proof outline is given.
    The first step is to observe that there exists a feasible solution to \eqref{eq:t-gospa_mot_reg} with the property that $M_{(i_1, j_1), \dots, (i_T, j_T)} > 0$ for all indices such that $C_{(i_1, j_1), \dots, (i_T, j_T)} < \infty$.
    The remaining results them follows by appropriate use of strong duality; see, e.g., \cite[Chp.~5]{boyd2004convex}.
\end{proof}

Based on the proposition above, an algorithm to solve \eqref{eq:t-gospa_mot_reg} is now constructed as a coordinate ascent algorithm to solve \eqref{eq:t_gospa_dual_reg}.
To simplify the presentation, we introduce the transformed dual variables $\bar{u}^t = \exp ( \bar{\lambda}^t / \varepsilon )$ and $\tilde{u}^t = \exp ( \tilde{\lambda}^t / \varepsilon )$, and the tensor $K = \exp \left( -C / \varepsilon \right)$.
Next, let $M^{(\bar{\lambda},\tilde{\lambda})}$ be a tensor of the form \eqref{eq:primal_sol}, and
let $\bar{w}_{i_{\tau}}^{\tau} = \big(P_{\tau}(M^{(\bar{\lambda},\tilde{\lambda})}) \mathbf{1}_{n+1}\big)_{i_{\tau}} \Big/ \bar{u}_{i_{\tau}}^{\tau}$. We observe that
\begin{equation*}
\begin{multlined}[\columnwidth]
    \bar{w}_{i_{\tau}}^{\tau}
        = \sum_{j_t}
            \sum_{(i_1, j_1), \dots, (i_T, j_T) \setminus (i_t, j_t)} \!\!\!\!\!\!\!\! K_{(i_1, j_1), \dots, (i_T, j_T)}
            \prod_{\substack{t = 1 \\ t \neq \tau}}^T \bar{u}^t_{i_t} \prod_{t=1}^T \tilde{u}^t_{j_t}
\end{multlined}
\end{equation*}
is independent of $\bar{u}_{i_{\tau}}^{\tau}$. Now, since problem \eqref{eq:t_gospa_dual_reg} is convex, in order to maximize $\phi(\bar{\lambda}, \tilde{\lambda})$ with respect to $\lambda_{i_{\tau}}^{\tau}$, we take the derivative of the function and set it to zero.
This gives
\begin{align*}
    0 = \frac{\partial \phi}{\partial \lambda_{i_{\tau}}^{\tau}}
        = \bar{\mu}_{i_{\tau}} -\bar{u}_{i_{\tau}}^{\tau} \bar{w}_{i_{\tau}}^{\tau},
\end{align*}
from which we obtain the update rule
\begin{equation} \label{eq:u_bar_vec_it}
    \bar{u}^\tau \gets \bar{\mu} \oslash P_{\tau}(M^{(\bar{\lambda},\tilde{\lambda})}) \mathbf{1}_{n+1} \odot \bar{u}^\tau, \quad \tau = 1, \dots, T,
\end{equation}
where $\oslash$ and $\odot$ denotes elementwise division and multiplication, respectively.
Analogously, we obtain the update rule
\begin{equation} \label{eq:u_tilde_vec_it}
    \tilde{u}^\tau \gets \tilde{\mu} \oslash P_{\tau}(M^{(\bar{\lambda},\tilde{\lambda})})^\top \mathbf{1}_{m+1} \odot \tilde{u}^\tau,
    \quad \tau = 1, \dots, T.
\end{equation}

\subsection{Efficient Computation of Projections of Transport Tensors and Convergence of the Sinkhorn Algorithm} \label{sub_sec:efficent_proj}

Iteratively performing the updates \eqref{eq:u_bar_vec_it} and
\eqref{eq:u_tilde_vec_it} is
coordinate-wise ascent in the dual problem \eqref{eq:t_gospa_dual_reg}, i.e., it is a Sinkhorn algorithm for solving \eqref{eq:t-gospa_mot_reg}.
However, the Sinkhorn iterations only partially alleviate the computational cost in the multimarginal setting. More specifically,
\eqref{eq:u_bar_vec_it} and \eqref{eq:u_tilde_vec_it}
require the computation of $P_t(M^{(\bar{\lambda},\tilde{\lambda})})$.
The latter involves nested sums which, if evaluated naively, requires $\mathcal{O}(m^T n^T)$ operations to evaluate.
To derive a method for efficiently computing $P_t(M^{(\bar{\lambda},\tilde{\lambda})})$, let $k^t_{(i_t,j_t)} = \exp( -D^t_{(i_t,j_t)} / \varepsilon )$ and
$\hat{k}_{(i_t, j_t), (i_{t+1}, j_{t+1})} = \exp(- F_{(i_t,j_t), (i_{t+1},j_{t+1})} / \varepsilon )$,
and note that by the structure of the cost in \eqref{eq:t-gospa_mot-cost}, the tensor $K$ can be factorized as
\begin{align*}
    K_{(i_1, j_1), \dots, (i_T, j_T)}
    = \prod_{t=1}^T k^t_{(i_t,j_t)} \prod_{t=1}^{T-1} \hat{k}_{(i_t, j_t), (i_{t+1}, j_{t+1})}.
\end{align*}
Using this structure, we have the following result.
\begin{theorem}\label{thm:proj_comp}
$P_{\tau}(M^{(\bar{\lambda},\tilde{\lambda})})_{(i_{\tau}, j_{\tau})}$ can be written as
\begin{equation} \label{eq:alg_2_proj_decomp}
    \begin{split}
        P_{\tau}(M^{(\bar{\lambda},\tilde{\lambda})})_{(i_{\tau}, j_{\tau})}
        = \rarrow{\Phi}^\tau_{(i_\tau, j_\tau)}
          k^\tau_{(i_\tau, j_\tau)}
          \bar{u}^\tau_{i_\tau}
          \tilde{u}^\tau_{j_\tau}
          \larrow{\Phi}^\tau_{(i_\tau, j_\tau)},
    \end{split}
\end{equation}
where
$\rarrow{\Phi}^\tau_{(i_\tau, j_\tau)}$ and $\larrow{\Phi}^\tau_{(i_\tau, j_\tau)}$ are defined recursively as
\begin{equation} \label{eq:alg_2_phi_hat_update}
\begin{aligned} 
    \rarrow{\Phi}^\tau_{(i_\tau, j_\tau)}
    &= \sum_{(i_{\tau - 1}, j_{\tau - 1})}
        \rarrow{\Phi}^{\tau - 1}_{(i_{\tau - 1}, j_{\tau - 1})} \\
        & \quad {} \cdot k^{\tau - 1}_{(i_{\tau - 1}, j_{\tau - 1})}
        \hat{k}_{(i_{\tau - 1}, j_{\tau - 1}), (i_\tau, j_\tau)}
        \bar{u}^{\tau - 1}_{i_{\tau - 1}}
        \tilde{u}^{\tau - 1}_{j_{\tau - 1}},
\end{aligned}
\end{equation}
with $\rarrow{\Phi}^1_{(i_1, j_1)} = 1$, and 
\begin{equation} \label{eq:alg_2_phi_update}
\begin{aligned} 
    \larrow{\Phi}^\tau_{(i_\tau, j_\tau)}
    &= \sum_{(i_{\tau + 1}, j_{\tau + 1})}
        \larrow{\Phi}^{\tau + 1}_{(i_{\tau + 1}, j_{\tau + 1})} \\
        & \quad {} \cdot k^{\tau + 1}_{(i_{\tau + 1}, j_{\tau + 1})}
        \hat{k}_{(i_\tau, j_\tau), (i_{\tau + 1}, j_{\tau + 1})}
        \bar{u}^{\tau + 1}_{i_{\tau + 1}}
        \tilde{u}^{\tau + 1}_{j_{\tau + 1}}.
\end{aligned} 
\end{equation}
with $\larrow{\Phi}^T_{(i_T, j_T)} = 1$.    
\end{theorem}
\begin{proof}
    With the definitions in \eqref{eq:alg_2_phi_hat_update} and \eqref{eq:alg_2_phi_update}, it can be readily shown that \eqref{eq:alg_2_proj_decomp} holds. Details are omitted.
\end{proof}

The structure in \eqref{eq:alg_2_proj_decomp}--\eqref{eq:alg_2_phi_update} means that we have a way of computing the projections in $\mathcal{O}(m^2 n^2)$ operations 
instead of $\mathcal{O}(m^T n^T)$ operations,%
\footnote{This can further be reduced to $\mathcal{O}(\max(m, n) \min(m, n)^2)$ by utilizing that infinite elements in $F$ give zero elements in $\hat{k}$.}
greatly improving the computational speed of the updates \eqref{eq:u_bar_vec_it} and \eqref{eq:u_tilde_vec_it}.
Finally, the results from Section \ref{sub_sec:block_coord_asc} and Theorem~\ref{thm:proj_comp}
yields Algorithm~\ref{alg:sinkhorn_t_gospa}.
By using \cite{luo_convergence_1992}, we get that the algorithm is convergent in the following sense (cf.~\cite[Prop.~1]{haasler_scalable_2023}).

\begin{theorem}\label{thm:convergence}
    Algorithm \ref{alg:sinkhorn_t_gospa} converges
    linearly to an optimal solution of \eqref{eq:t_gospa_dual_reg}.
\end{theorem}

\begin{remark}\label{rem:computational_complexity}
From the discussion above, we note that one full iteration in Algorithm~\ref{alg:sinkhorn_t_gospa} requires $\mathcal{O}(T m^2 n^2)$ operations.
While Theorem~\ref{thm:convergence} shows that Algorithm~\ref{alg:sinkhorn_t_gospa} converges linearly, it neither gives an explicit rate in terms of the parameters of the problem, nor a bound on the number of iterations needed in order to get close to the optimal solution.
Such bounds are known to exist for some Sinkhorn methods, see, e.g., \cite{altschuler2023polynomial, lin2022complexity, fan2022complexity}, \cite[Sec.~4.2]{peyre2019computational}.
Deriving such bounds for this problem would be interesting, but it is left for future work.
\end{remark}

\begin{algorithm}[tbh]
    \caption{Block Coordinate Ascent Method for \eqref{eq:t_gospa_dual_reg}}
    \label{alg:sinkhorn_t_gospa}
    \begin{algorithmic}[1]
        \State Initialize $\bar{u}$ and $\tilde{u}$ to one everywhere
        \While{Not converged}
            \State $\larrow{\Phi}^T \gets 1$
            \For{$t = T - 1, \dots, 1$}
                \State Compute $\larrow{\Phi}^t$ from $\larrow{\Phi}^{t+1}$ using \eqref{eq:alg_2_phi_update}
            \EndFor
            \State $\rarrow{\Phi}^1 \gets 1$
            \For{$t = 1, \dots, T-1$}
                \State $\bar{u}^t \gets \bar{\mu} \oslash P_t(M^{(\bar{\lambda},\tilde{\lambda})}) \mathbf{1}_{n+1} \odot \bar{u}^t$, \text{ using \eqref{eq:alg_2_proj_decomp}}
                \State $\tilde{u}^t \gets \tilde{\mu} \oslash P_t(M^{(\bar{\lambda},\tilde{\lambda})})^\top \mathbf{1}_{m+1} \odot \tilde{u}^t$, \text{ using \eqref{eq:alg_2_proj_decomp}}
                \State Compute $\rarrow{\Phi}^{t+1}$ from $\rarrow{\Phi}^t$ using \eqref{eq:alg_2_phi_hat_update}
            \EndFor
            \State $\bar{u}^T \gets \bar{\mu} \oslash P_T(M^{(\bar{\lambda},\tilde{\lambda})}) \mathbf{1}_{n+1} \odot \bar{u}^T$
            \State $\tilde{u}^T \gets \tilde{\mu} \oslash P_T(M^{(\bar{\lambda},\tilde{\lambda})})^\top \mathbf{1}_{m+1} \odot \tilde{u}^T$
        \EndWhile
    \end{algorithmic}
\end{algorithm}

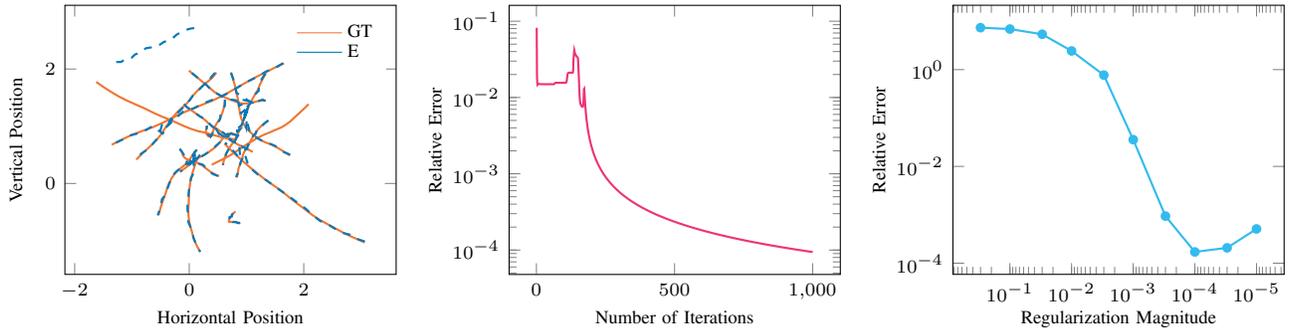
\begin{figure*}
    \centering
    \begin{tikzpicture}
        \begin{groupplot}[group style={group size=3 by 1, horizontal sep=15mm}, width=0.33\linewidth]
            \nextgroupplot[            
                xlabel={Horizontal Position},
                ylabel={Vertical Position},
                x label style={at={(axis description cs:0.5,-0.1)}, anchor=north},
                y label style={at={(axis description cs:-0.1,.5)}, anchor=south},
                legend pos=north east,
                legend style={row sep=-1mm, column sep=0mm, draw=none},
                legend cell align=left,
                axis equal,
            ]
            \foreach \x [evaluate=\x as \y using int(\x+1)] in {0,2,...,30}
            {
                \addplot[thick, tol_vib_orange, forget plot] table [x index=\y, y index=\x]
                    {data/example_data_ground_truth.tsv};
            }
            \addlegendimage{tol_vib_orange}
            \addlegendentry{GT}

            \foreach \x [evaluate=\x as \y using int(\x+1)] in {0,2,...,28}
            {
                \addplot[thick, dashed, tol_vib_blue, forget plot] table [x index=\y, y index=\x]
                    {data/example_data_estimate.tsv};
            }
            \addlegendimage{tol_vib_blue}
            \addlegendentry{E}
            
            \nextgroupplot[            
                xlabel={Number of Iterations},
                ylabel={Relative Error},
                x label style={at={(axis description cs:0.5,-0.1)}, anchor=north},
                y label style={at={(axis description cs:-0.175,.5)}, anchor=south},
                ymode=log,
                legend pos=north east,
                legend style={row sep=-1mm, column sep=0mm, draw=none},
                legend cell align=left,
            ]
            \addplot[thick, tol_vib_magenta] table [x index=0, y index=1] {data/example_convergence_cost.tsv};

            \nextgroupplot[            
                xlabel={Regularization Magnitude},
                ylabel={Relative Error},
                x label style={at={(axis description cs:0.5,-0.1)}, anchor=north},
                y label style={at={(axis description cs:-0.175,.5)}, anchor=south},
                ymode=log,
                xmode=log,
                x dir=reverse,
                legend pos=north east,
                legend style={row sep=-1mm, column sep=0mm, draw=none},
                legend cell align=left,
            ]
            \addplot[thick, tol_vib_cyan, mark=*, mark size=0.5mm] table [x index=0, y index=1]
                {data/example_regularization_analysis.tsv};
        \end{groupplot}
    \end{tikzpicture} 
    \caption{
        \textbf{Left:} Simulated data.
        GT is the simulated ground truth data, and E is the simulated estimated trajectories.
        \textbf{Middle:} Convergence plot of Algorithm \ref{alg:sinkhorn_t_gospa} on the data in
            Figure \ref{fig:misc} (left).
        \textbf{Right:} Relative error depending on the
        parameter $\eta$.
    }
    \label{fig:misc}
\end{figure*}

%%%%%%%%%%%%%%%%%%%%%%%%%%%%%%%%%%%%%%%%%%%%%%%%%%%%%%%%%%%%%%%%%%%%%%%%%%%%%%%%%%%%%%%%%%%%%%%%%%%%%%%%%%%%%%%%%%%%%%%%

\section{Numerical Results} \label{sec:numerical_results}

To demonstrate the efficiency of Algorithm~\ref{alg:sinkhorn_t_gospa}, we test it on simulated data. We compare the algorithm with four off-the-shelf LP solvers, namely the well-known commercial solvers Gurobi, CPLEX and MOSEK, as well as the open source solver CLP \cite{forrest_coin-orclp_2023}.
For all simulations we use $c = 0.25$, $p = 1$, $\gamma = 1$ in the definition of the TGOSPA metric. Moreover, 
the value of $\varepsilon$ needs to be selected in relation to the magnitude of the rest of the objective
function \eqref{eq:t-gospa_mot_reg_obj}. For this reason, we take $\varepsilon = \eta \cdot T \cdot \max(\max(D), \max(F))$,%
\footnote{The maximum over $F$ is taken over all finite values of the tensor.}
where $\eta$ becomes the scaled
regularization parameter. For stability reasons, we use the log-sum-exp rewriting and do all computations in the log-domain (see, e.g., \cite[Sec.~4.4]{peyre2019computational}, \cite{schmitzer2019stabilized}).

Data is generated using the algorithm for structured data described in \cite[Sec.~5.1.1]{nevelius2024efficient}.
Figure \ref{fig:misc} (left) shows an example of such data, simulated with parameters $m_t = 14$, $m_f = 2$, $n_f = 1$, $T = 20$, $r = 1$, $q = 0.9$, $c_s = 0.25$, $N_\mathrm{ts} = 20$, $N_\mathrm{max} = 10^5$, and $\sigma = 0.01$.
In the below results,
we use the relative error $|f_{\text{opt}} - f|/f_{\text{opt}}$ as a measure of accuracy. Here,
$f_{\text{opt}}$ is the optimal value of the LP relaxation of \eqref{eq:t-gospa}, and $f = \langle C, M_{\text{iter}} \rangle$ where $M_{\text{iter}}$ is the approximate transport plan either at the current iterate of Algorithm~\ref{alg:sinkhorn_t_gospa} or after termination of the algorithm. Which of the two interpretations of $f$ that is used is clear from context.
The value $f_{\text{opt}}$ is computed using an LP solver.

Our method also allows for parallelization, and can thus be accelerated by running it on a GPU. The accelerated version is also included in the comparisons, and it should be noted that LP solvers can in general not be accelerated in the same way. The experiments were run on a AMD Ryzen 5 5600X 3701 MHz CPU and a NVIDIA GeForce GTX 1060 GPU. Our method was implemented using PyTorch, and the LP solvers were interfaced using the PuLP package in Python.

\subsection{Illustrating Convergence and Effect of Regularization}

To illustrate the behavior of the algorithm, we first apply it to the simulated data shown in Figure \ref{fig:misc} (left).
For $\eta = 10^{-5}$, the relative error as a function of the number of iterations of Algorithm \ref{alg:sinkhorn_t_gospa}
is shown in Figure \ref{fig:misc} (middle). We note that most of the improvement
happens in a transient phase, suggesting that early stopping should be considered when computational
resources are limited.

Figure \ref{fig:misc} (right) shows the effect
of $\eta$
on the result of 
Algorithm \ref{alg:sinkhorn_t_gospa}.
In the simulations, we terminate the algorithm when the relative step size
$|| \bm{u} - \bm{u}_\text{prev} || / || \bm{u}_\text{prev} ||$ is less than $1.5 \cdot 10^{-6}$. Here, $\bm{u}$ and
$\bm{u}_\text{prev}$ are vectors constructed by concatenating all transformed dual variables from two consecutive iterations.
As expected, small values of $\eta$ leads to good results, and decreasing $\eta$ generally leads to better approximations.
The latter is true except for very small $\eta$.
We believe that the this is due to that the convergence speed is expected to decrease with $\varepsilon$, i.e., with $\eta$ (cf.~Remark~\ref{rem:computational_complexity}), but that we still use the same breaking criteria in all cases.

\subsection{Computational Efficiency}

In order to evaluate the computational efficiency of Algorithm~\ref{alg:sinkhorn_t_gospa}, we vary the size of the input data and compare the
wall-clock running time to those of the LP solvers.
Noted that the LP solvers compute the exact value of the
LP-relaxed TGOSPA metric, while our algorithm only provides an approximation. For this reason, we also record the relative
error compared to the exact solution. Using the data generation procedures
with
parameters $m_t = m - m_f$, $m_f = 0.1m$ (rounded to the nearest integer), $n_f = m_f$, $r = 1$, $q = 0.9$, $c_s = 0.25$,
$N_\mathrm{ts} = m$, $N_\mathrm{max} = 100m$, and $\sigma = 0.01$, we generate
random scenarios for varying number
of trajectories and time steps.
For the case of varying number of trajectories, we fix $T = 100$, and for varying number of time
steps, we fix $m = 100$.
We use $\eta = 10^{-4}$, and terminate the algorithm when the relative step size is less than $10^{-3}$.

For each set of simulation parameters we generate $25$ scenarios.
Averaged results over the scenarios are shown in Figure \ref{fig:timings}.
Among the LP solvers, Gurobi was the fastest, and therefore it is the only one that was run for the largest scenarios. For example, CLP was not able to solve even the smallest scenario in the bottom row of Figure \ref{fig:timings}
in under two minutes.
We see that for larger tracking
scenarios, our algorithm computes solutions, with a mean relative error of around $1 \%$, significantly faster than the LP solvers. 

\pgfplotsset{
    /pgfplots/legend image code/.code={%
        \draw[mark repeat=1,mark phase=2,#1] 
            plot coordinates {
                (0cm,0cm) 
                (0.15cm,0cm)
                (0.3cm,0cm)
            };
    }
}

\begin{figure}
    \centering
    \begin{tikzpicture}
        \begin{groupplot}[group style={group size=2 by 2, horizontal sep=10mm, vertical sep=12mm}, width=0.57\columnwidth]
            \nextgroupplot[            
                xlabel={Number of Targets},
                ylabel={Wall-Clock Time (Seconds)},
                x label style={at={(axis description cs:0.5,-0.1)}, anchor=north},
                y label style={at={(axis description cs:-0.15,.5)}, anchor=south},
                no marks,
                legend style={at={(1.1, 1.22)}, row sep=-1mm, column sep=0mm, draw=none, nodes={scale=0.75, transform shape}, line width = 0.1mm},
                legend columns=3,
                legend cell align=left,
                ymode=log,
            ]
            \addplot[thick, tol_vib_orange] table [x index=0, y index=1] {data/timings_targets.tsv};
            \addlegendentry{Gurobi}
            \addplot[thick, tol_vib_blue] table [x index=0, y index=2] {data/timings_targets.tsv};
            \addlegendentry{CPLEX}
            \addplot[thick, tol_vib_cyan] table [x index=0, y index=3] {data/timings_targets.tsv};
            \addlegendentry{MOSEK}
            \addplot[thick, tol_vib_teal] table [x index=0, y index=4] {data/timings_targets.tsv};
            \addlegendentry{CLP}
            \addplot[thick, tol_vib_magenta] table [x index=0, y index=5] {data/timings_targets.tsv};
            \addlegendentry{Ours (CPU)}
            \addplot[thick, tol_vib_magenta, dashed] table [x index=0, y index=6] {data/timings_targets.tsv};
            \addlegendentry{Ours (GPU)}
            
            \nextgroupplot[            
                xlabel={Number of Targets},
                ylabel={Relative Error},
                x label style={at={(axis description cs:0.5,-0.1)}, anchor=north},
                y label style={at={(axis description cs:-0.15,.5)}, anchor=south},
                legend pos=north east,
                legend style={row sep=-1mm, column sep=0mm, draw=none},
                legend cell align=left,
                legend columns=2,
                ymax=0.04,
            ]
            \addplot[thick, tol_vib_teal] table [x index=0, y index=1] {data/error_targets.tsv};
            \addlegendentry{Mean}
            \addplot[thick, tol_vib_red] table [x index=0, y index=2] {data/error_targets.tsv};
            \addlegendentry{Max}
            \addplot[name path=low_err, draw=none] table [x index=0, y index=3] {data/error_targets.tsv};
            \addplot[name path=high_err, draw=none] table [x index=0, y index=4] {data/error_targets.tsv};
            \addplot[color=tol_vib_teal!25] fill between[of = low_err and high_err];

            \nextgroupplot[            
                xlabel={Number of Time Steps},
                ylabel={Wall-Clock Time (Seconds)},
                x label style={at={(axis description cs:0.5,-0.1)}, anchor=north},
                y label style={at={(axis description cs:-0.15,.5)}, anchor=south},
                no marks,
                legend style={at={(1.1, 1.22)}, row sep=-1mm, column sep=0mm, draw=none, nodes={scale=0.75, transform shape}, line width = 0.1mm},
                legend columns=3,
                legend cell align=left,
                ymode=log,
            ]

            %\definecolor{tol_vib_orange}{HTML}{EE7733}
            %\definecolor{tol_vib_blue}{HTML}{0077BB}
            %\definecolor{tol_vib_cyan}{HTML}{33BBEE}
            %\definecolor{tol_vib_magenta}{HTML}{EE3377}
            %\definecolor{tol_vib_red}{HTML}{CC3311}
            %\definecolor{tol_vib_teal}{HTML}{009988}
            %\definecolor{tol_vib_gray}{HTML}{BBBBBB}
            
            \addplot[thick, tol_vib_orange] table [x index=0, y index=1] {data/timings_time_steps.tsv};
            \addlegendentry{Gurobi}
            \addplot[thick, tol_vib_blue] table [x index=0, y index=2] {data/timings_time_steps.tsv};
            \addlegendentry{CPLEX}
            \addplot[thick, tol_vib_cyan] table [x index=0, y index=3] {data/timings_time_steps.tsv};
            \addlegendentry{MOSEK}
            \addplot[thick, tol_vib_magenta] table [x index=0, y index=4] {data/timings_time_steps.tsv};
            \addlegendentry{Ours (CPU)}
            \addplot[thick, tol_vib_magenta, dashed] table [x index=0, y index=5] {data/timings_time_steps.tsv};
            \addlegendentry{Ours (GPU)}
            
            \nextgroupplot[            
                xlabel={Number of Time Steps},
                ylabel={Relative Error},
                x label style={at={(axis description cs:0.5,-0.1)}, anchor=north},
                y label style={at={(axis description cs:-0.15,.5)}, anchor=south},
                legend pos=north east,
                legend style={row sep=-1mm, column sep=0mm, draw=none},
                legend cell align=left,
                legend columns=2,
                ymax=0.055,
            ]
            \addplot[thick, tol_vib_teal] table [x index=0, y index=1] {data/error_time_steps.tsv};
            \addlegendentry{Mean}
            \addplot[thick, tol_vib_red] table [x index=0, y index=2] {data/error_time_steps.tsv};
            \addlegendentry{Max}
            \addplot[name path=low_err, draw=none] table [x index=0, y index=3] {data/error_time_steps.tsv};
            \addplot[name path=high_err, draw=none] table [x index=0, y index=4] {data/error_time_steps.tsv};
            \addplot[color=tol_vib_teal!25] fill between[of = low_err and high_err];
        \end{groupplot}
    \end{tikzpicture} 
    \caption{
        Average runtimes
        and errors for Algorithm \ref{alg:sinkhorn_t_gospa} on varying data sizes.
        The top row shows results for fixed number of time steps $T$ and varying number of targets $m$. The bottom row shows results for varying number of time steps $T$ and fixed number of targets $m$.
        The shaded regions in the right column is 
        one standard deviation from the mean.
    }
    \label{fig:timings}
\end{figure}
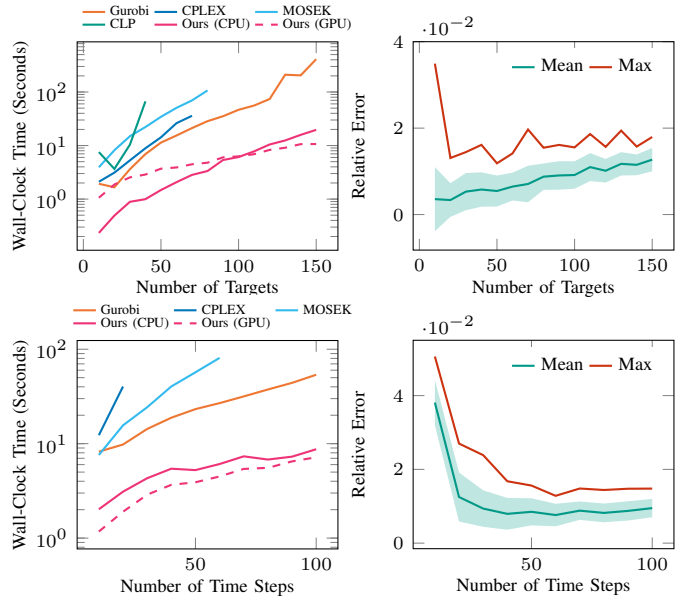

%%%%%%%%%%%%%%%%%%%%%%%%%%%%%%%%%%%%%%%%%%%%%%%%%%%%%%%%%%%%%%%%%%%%%%%%%%%%%%%%%%%%%%%%%%%%%%%%%%%%%%%%%%%%%%%%%%%%%%%%

\section{Conclusions} \label{sec:conclusions}

In this paper, we present an algorithm for fast approximation of the LP-relaxed TGOSPA metric. The algorithm is derived by reformulating the TGOSPA problem as an unbalanced multimarginal optimal transport problem,
and leveraging ideas from the literature on entropy regularized optimal transport.
Numerical results show that the method provides adequate approximations of the metric while significantly reducing computational costs.
An interesting future direction for research would be to
explore the possibility of utilizing the differentiability of Algorithm~\ref{alg:sinkhorn_t_gospa} (in the sense of reverse mode of automatic differentiation) to compute gradients with respect to the elements of the cost tensor, which would open
a path toward data-driven MTT algorithms.

\bibliographystyle{plain}
\bibliography{}

\end{document}

%% file: figures/structure_figure.tex
\begin{tikzpicture}
    [
        marginal/.style={circle, draw=black!75, fill=tol_vib_gray!20, thick, inner sep=0pt, minimum size=25mm, node font=\Large},
        empty/.style={regular polygon, regular polygon sides=4, draw=black, fill=black!5, minimum size=8.5mm},
        dot/.style={anchor=base, draw=none, fill=black!10, text depth=.5ex, text height=2ex, text width=1em},
        filled/.style={anchor=base, draw=black, fill=tol_vib_red!60, text depth=.5ex, text height=2ex, text width=1em},
        between base/.style args={#1 and #2}{between=#1.base and #2.base},
        square/.style={regular polygon, regular polygon sides=4}
    ]
   
    \node[marginal] (1)              {};
    \node[marginal] (2) [right=of 1] {};
    \node[marginal] (3) [right=of 2] {};
    \node[marginal] (4) [right=of 3] {};
    \node[marginal] (5) [right=of 4] {};
    
    \draw [->] (1) to[bend left=45] node[midway, sloped, above]{\Large $P_{1, 2}(M)$}     (2);
    \draw [->] (2) to[bend left=45] node[midway, sloped, above]{\Large $P_{2, 3}(M)$}     (3);
    \draw [->] (3) to[bend left=45] node[midway, sloped, above]{\Large $P_{T-2, T-1}(M)$} (4);
    \draw [->] (4) to[bend left=45] node[midway, sloped, above]{\Large $P_{T-1, T}(M)$}   (5);

    \matrix (m1) [matrix of nodes, below left = 15mm and -17mm of 1, row sep=-\pgflinewidth, column sep=-\pgflinewidth, nodes=empty, nodes in empty cells]{
        |[fill=tol_vib_red!75]| &                         &                         &                         &                   \\
                                & |[fill=tol_vib_red!75]| &                         &                         &                   \\
                                &                         & |[fill=tol_vib_red!75]| &                         &                   \\
                                &                         &                         & |[fill=tol_vib_red!75]| & |[fill=black!75]| \\
    };
    \node[text height=10mm] at (m1.south) {\Large $P_1(M)$};

    \matrix (m2) [matrix of nodes, below right = 15mm and -17mm of 2, row sep=-\pgflinewidth, column sep=-\pgflinewidth, nodes=empty, nodes in empty cells]{
        |[fill=tol_vib_red!75]| &                         &                         &                         &                   \\
                                & |[fill=tol_vib_red!75]| &                         &                         &                   \\
                                &                         & |[fill=tol_vib_red!75]| &                         &                   \\
                                &                         &                         & |[fill=tol_vib_red!75]| & |[fill=black!75]| \\
    };
    \node[text height=10mm] at (m2.south) {\Large $P_2(M)$};
    
    \matrix (m4) [matrix of nodes, below left = 15mm and -17mm of 4, row sep=-\pgflinewidth, column sep=-\pgflinewidth, nodes=empty, nodes in empty cells]{
                                & |[fill=tol_vib_red!75]| &                         &                         &                  \\
        |[fill=tol_vib_red!75]| &                         &                         &                         &                  \\
                                &                         &                         & |[fill=tol_vib_red!75]| &                  \\
                                &                         & |[fill=tol_vib_red!75]| &                         & |[fill=black!75]| \\
    };
    \node[text height=10mm] at (m4.south) {\Large $P_{T-1}(M)$};

    \matrix (m5) [matrix of nodes, below right = 15mm and -17mm of 5, row sep=-\pgflinewidth, column sep=-\pgflinewidth, nodes=empty, nodes in empty cells]{
                                &                         &                         &                         & |[fill=tol_vib_red!75]| \\
                                &                         &                         &                         & |[fill=tol_vib_red!75]| \\
                                &                         &                         &                         & |[fill=tol_vib_red!75]| \\
        |[fill=tol_vib_red!75]| & |[fill=tol_vib_red!75]| & |[fill=tol_vib_red!75]| & |[fill=tol_vib_red!75]| &                         \\
    };
    \node[text height=10mm] at (m5.south) {\Large $P_T(M)$};

    \draw[black!50, ->]             (m1-1-5.east) to[bend left=45] node[midway, sloped, above]      {} (m2-1-1.west);
    \draw[black, thick, ->]         (m1-2-5.east) to[bend left=45] node[midway, sloped, above] (e1) {} (m2-2-1.west);
    \draw[black!50, ->]             (m1-3-5.east) to[bend left=45]                                     (m2-3-1.west);
    \draw[black!50, ->]             (m1-4-5.east) to[bend left=45] node[midway, sloped, above] (ea) {} (m2-4-1.west);
    \draw[draw=none, ->]            (m1-4-5.east) to               node[midway, sloped, above] (ea) {} (m2-4-1.west);

    \draw[black!50, ->]             (m4-1-5.east) to[bend left=45] node[midway, sloped, above]      {} (m5-1-1.west);
    \draw[black!50, ->]             (m4-2-5.east) to[bend left=45] node[midway, sloped, above]      {} (m5-2-1.west);
    \draw[black!50, ->]             (m4-3-5.east) to[bend left=45]                                     (m5-3-1.west);
    \draw[black, thick, ->]         (m4-4-5.east) to[bend left=45] node[midway, sloped, above] (e2) {} (m5-4-1.west);
    \draw[draw=none, thick, ->]     (m4-4-5.east) to node[midway, sloped, above]               (eb) {} (m5-4-1.west);

    \draw[black!50, ->, dashed]     (m2-1-5.east) to[bend left=45] node[midway, sloped, above] (es) {} (m4-1-1.west);
    \draw[black!50, ->, dashed]     (m2-2-5.east) to[bend left=45] node[midway, sloped, above]      {} (m4-2-1.west);
    \draw[black!50, ->, dashed]     (m2-3-5.east) to[bend left=45] node[midway, sloped, above] (e3) {} (m4-3-1.west);
    \draw[black!50, ->, dashed]     (m2-4-5.east) to[bend left=45] node[midway, sloped, above]      {} (m4-4-1.west);
    \draw[draw=none, thick, ->]     (m2-4-5.east) to node[midway, sloped, above]               (ec) {} (m4-4-1.west);
    
    \matrix (t1) [matrix of nodes, below = 15mm of ea, row sep=-\pgflinewidth, column sep=-\pgflinewidth, nodes=empty, nodes in empty cells]{
        & & & & \\
        & |[fill=tol_vib_teal!75]| & & &\\
        & & & & \\
        & & & & \\
        & & & & \\
    };

    \matrix (t2) [matrix of nodes, below = 15mm of eb, row sep=-\pgflinewidth, column sep=-\pgflinewidth, nodes=empty, nodes in empty cells]{
        & & & & \\
        & & & &\\
        & & |[fill=tol_vib_teal!75]| & & \\
        & & & & \\
        |[fill=tol_vib_teal!75]|& |[fill=tol_vib_teal!75]| & & |[fill=tol_vib_teal!75]| & \\
    };

    \node[square, draw=black, fill=black!5, minimum size=42.5mm] (c) [below = 16.45mm of ec] {\Large $\cdots$};
    
    \node[text height=10mm] at (t1.south) {\Large $P_{1, 2}(M)_{(2, :), (2, :)}$};
    \node[text height=10mm] at (t2.south) {\Large $P_{T-1, T}(M)_{(4, :), (4, :)}$};

    \draw[black, thick, ->] (e1) to (t1.north);
    \draw[black, thick, ->] (e2) to (t2.north);
    \draw[black!50, ->, dashed] (ec) to ([yshift=(2.5)]c.north);
    
    \draw[black, ->, dashed] (1.center) to (m1.north);
    \draw[black, ->, dashed] (2.center) to (m2.north);
    \draw[black, ->, dashed] (4.center) to (m4.north);
    \draw[black, ->, dashed] (5.center) to (m5.north);
    
    \draw[black!50, ->, dashed] (3.center) to (es);

    \node[marginal] (1)              {$P_1(M)$};
    \node[marginal] (2) [right=of 1] {$P_2(M)$};
    \node[marginal] (3) [right=of 2] {$\cdots$};
    \node[marginal] (4) [right=of 3] {$P_{T-1}(M)$};
    \node[marginal] (5) [right=of 4] {$P_T(M)$};
\end{tikzpicture}